\documentclass[11pt]{amsart}

\makeatletter
\def\LaTeX{\leavevmode L\raise.42ex
    \hbox{\kern-.3em\size{\sfsize}{0pt}\selectfont A}\kern-.15em\TeX}
\makeatother

\newcommand{\BibTeX}{{\rm B\kern-.05em{\sc
          i\kern-.025emb}\kern-.08em\TeX}}

\makeatletter
\label{e:dispaa}
\label{e:dispau}
\label{e:dispav}
\label{e:dispaw}
\label{e:dispax}
\makeatother

\newtheorem{theorem}{Theorem}
\newtheorem{definition}{Definition}
\newtheorem{corollary}{Corollary}

\newtheorem{proposition}{Proposition}
\newtheorem{lemma}{Lemma}

\numberwithin{theorem}{section}
\numberwithin{corollary}{section}
\numberwithin{proposition}{section}
\numberwithin{lemma}{section}
\numberwithin{equation}{section}
\numberwithin{remark}{section}
\numberwithin{definition}{section}

\usepackage{latexsym}
\usepackage{amsmath}
\usepackage{amsfonts}
\usepackage{amssymb}
\usepackage{cite}
\usepackage{color}

\newcommand{\R}{\mathbb{R}}

\newcommand{\N}{\mathbb{N}}

\newcommand{\cC}{{\mathcal C}}

\newcommand{\cH}{{\mathcal H}}

\newcommand{\cK}{{\mathcal K}}
\newcommand{\cL}{{\mathcal L}}

\newcommand{\cN}{{\mathcal N}}

\newcommand{\cT}{{\mathcal T}}

\newcommand{\eps}{\varepsilon}

\newcommand{\sN}{{\!\stackrel{}{{\mbox{\tiny$N$}}}}}

\renewcommand{\phi}{\varphi}

\newenvironment{altproof}[1]
{\noindent
{\em Proof of {#1}}.}
{\nopagebreak\mbox{}\hfill $\Box$\par\addvspace{0.5cm}}

\DeclareMathOperator{\diag}{diag}

\DeclareMathOperator{\innt}{int}

\begin{document}

\title[Schr\"odinger system]{Existence and nonexistence of entire solutions for non-cooperative
  cubic elliptic systems}

\author{Hugo Tavares, Susanna Terracini, Gianmaria Verzini and Tobias Weth}

\email{}


\maketitle

\begin{abstract}
In this paper we deal with the cubic Schr\"odinger system 
\[
-\Delta u_i = \sum_{j=1}^n \beta_{ij}u_j^2 u_i,\qquad
u_1,\dots,u_n \ge 0 \qquad \text{in $\R^\sN$ }\ (N\leq 3),
\]
where $\beta=(\beta_{i,j})_{ij}$ is a symmetric matrix with real coefficients and $\beta_{ii}\geq 0$ for every $i=1,\ldots,n$. We analyse the
existence and nonexistence of nontrivial solutions in connection with the properties of the matrix $\beta$, and provide a complete characterization in dimensions $N=1,2$.
Extensions to more general power-type nonlinearities are given.
\end{abstract}

\section{Introduction and main results}
\setcounter{equation}{0}
\label{sec:introduction}

The purpose of the present paper is to analyze existence and
nonexistence of nontrivial solutions of the cubic elliptic system
\begin{equation}
  \label{eq:special-case}
-\Delta u_i = \sum_{j=1}^n \beta_{ij}u_j^2 u_i,\qquad
u_1,\dots,u_n \ge 0 \qquad \text{in $\R^\sN$}
\end{equation}
when $N\leq 3$. Here $(\beta_{ij})_{ij}$ is a symmetric $n \times n$-matrix with
real coefficients and nonnegative diagonal elements, i.e. $\beta_{ii}
\ge 0$ for $i=1,\dots,n$. We say that $u=(u_1,\dots,u_n)$ is a {\em nontrivial solution} of
(\ref{eq:special-case}) if $u_i \not \equiv 0$ for at least one $i \in
\{1,\dots,n\}$, which then implies that $u_i>0$ on $\R^\sN$ by the
maximum principle. In order to motivate our results on \eqref{eq:special-case}, let us first consider the single elliptic equation
\begin{equation}
  \label{eq:1}
 -\Delta u = u^{p-1},\quad u >0 \qquad \text{in $\R^N$.}
\end{equation}
It is well known that \eqref{eq:1} admits solutions if and only if $N\geq 3$ and $p\geq 2^*=2N/(N-2)$. The nonexistence in the complementary cases has been proved by Gidas and Spruck in \cite{gs_nonexistence}. In this paper we will show that for $n\geq 2$ the existence or nonexistence of solutions of \eqref{eq:special-case} depends in a subtle way on the coefficients $\beta_{ij}$.

In the case where $\beta_{ij} \ge 0$ for all $i,j$ and $\beta_{ii}>0$ for all $i$, the elliptic system behaves to a certain extent similarly as
the single equation (\ref{eq:1}) for $p=3$. In this case the system does
not admit nontrivial solutions in dimensions $N\le 3$ (where the cubic
nonlinearity is subcritical).
This follows from a more general nonexistence result of Reichel and
Zou \cite{rz} relying on the method of moving spheres. It is the purpose of the present paper to study the
non-cooperative case where the off-diagonal coefficients $\beta_{ij}$
may be negative and therefore methods based on the maximum principle,
like the method of moving spheres, do not apply. More precisely, we will analyze how $B=(\beta_{ij})_{ij}$ must differ from a
 matrix satisfying $\beta_{ij} \ge 0$ for all $i,j$ in order to allow nontrivial solutions of
(\ref{eq:special-case}). In the special case where $N\leq2$ (unidimensional or planar
problem) or $n=2$ (two-component problem) we will answer this question
completely by giving a necessary and sufficient matrix condition for
the solvability of (\ref{eq:special-case}), see Corollaries
\ref{coro:complete_result 1} and \ref{coro:complete_result 2}
below. By this we complement and extend a recent nonexistence result, which has been obtained for the
two-component problem in \cite{dancer.wei.weth:10}.

Nonexistence results in the whole space -- also called Liouville type
theorems -- for the equation (\ref{eq:1}) and the system
(\ref{eq:special-case}) play a crucial role in deriving a priori
bounds for a larger class of boundary value problems via the rescaling
method of Gidas and Spruck. In fact, in \cite{gs_apriori} Gidas and Spruck have used their nonexistence result for
(\ref{eq:1}) to deduce a priori bounds for solutions of equations of the type
\begin{equation}
\label{eq:3}
\left \{
  \begin{aligned}
 -\Delta u &= f(x,u),\:u>0&&\qquad \text{in $\Omega$,}\\
  u&=0 &&\qquad \text{on $\partial \Omega$,}
  \end{aligned}
\right.
\end{equation}
Here $\Omega \subset \R^n$ is a smooth domain, and it is assumed that
$\frac{f(x,t)}{t^{p-1}} \to h(x)$ uniformly in $x$ as $t \to \infty$ for some subcritical
exponent $p$, where $h
\in \cC(\overline \Omega)$. In the same spirit, Dancer, Wei and Weth
\cite{dancer.wei.weth:10} have obtained some a priori bounds for the class of systems \eqref{eq:special-case}
in the two components case
\begin{equation}
\label{eq:13}
\left \{
  \begin{aligned}
 &-\Delta u +\lambda_1 u = \beta_{11} u^3 + \beta_{12} u v^2 &&\quad \text{in } \Omega\\
 &-\Delta v + \lambda_2v =\beta_{22}v^3 + \beta_{12}u^2 v &&\quad \text{in } \Omega
  \end{aligned}
\right.
\qquad \quad
\begin{aligned}
&u,v\geq 0 &&\quad \text{in $\Omega$}\\
&u,v=0 &&\quad \text{on $\partial \Omega.$}
\end{aligned}
\end{equation}
Another class of Liouville type results for cubic systems has been
proved, under some global growth condition, in \cite{nttv:2010}, allowing to obtain uniform H\"older estimates
for the solutions of system (\ref{eq:13}), and of the more general version
\begin{equation}
  \label{eq:14}
 -\Delta u_i +\lambda_i u_i= \sum_{j=1}^n \beta_{ij}u_j^2 u_i \quad
\text{in $\Omega$}\qquad u_1=\dots=u_n=0\quad \text{ on }\partial \Omega.
\end{equation}
These nonlinear Schr\"odinger systems have received extensive attention in recent years, since they
appear in mathematical models for different phenomena in physics such as nonlinear optics and
Bose--Einstein condensation, see e.g. \cite{EGB,clll,sirakov,ac} and the references therein. In particular, for $\lambda_i > 0$, the question of which conditions on $(\beta_{ij})_{ij}$ assure
the existence of positive solutions has been widely
studied, see e.g. \cite{ac, linwei,mmp,bartschwang,ww1,ww2,sirakov}. For $\lambda_i<0$, existence and multiplicity of solutions in some particular cases were obtained also in \cite{norisramos}.
We remark that, in order to derive a priori bounds for (\ref{eq:14})
via the rescaling method of Gidas and Spruck, a nonexistence result is
needed both for the problem (\ref{eq:special-case}) and for
nonnegative nontrivial solutions of the half space problem
\begin{equation}
\label{eq:15}
 -\Delta u_i = \sum_{j=1}^n \beta_{ij}u_j^2   \quad
 \text{in $\R^\sN_+\;$ for $i=1,\dots,n$,}\qquad u_1=\dots=u_n=0 \; \text{ on }\partial \R^\sN_+,
\end{equation}
where $\R^\sN_+= \{x \in \R^\sN\::\: x_\sN >0\}$. There is strong
evidence that the nonexistence of nontrivial
solutions of (\ref{eq:special-case}) -- for a certain matrix $B$ -- also gives rise to
nonexistence of nontrivial nonnegative solutions of (\ref{eq:15}). In the two
component case $n=2$, this was already observed in
\cite{dancer.wei.weth:10}, but the argument in that paper does not
extend to the case $n \ge 3$. Since problem (\ref{eq:15}) requires
very different techniques, it will be treated in
future work, see  \cite{dancer.weth:10}.

To state our results, we first need to recall some notions for symmetric matrices. So let
$S(n)$ denote the space of
symmetric $n \times n$-matrices with real coefficients, and let $C^n_{+} \subset \R^n$ denote the closed
cone of all $c \in \R^n$ with nonnegative components $c_i$.
A matrix $B=(\beta_{ij})_{ij} \in S(n)$ is called\\
$\bullet$ {\em positive semidefinite (resp. positive definite)} if
$$
\sum
  \limits_{i,j=1}^n \beta_{ij}c_i c_j \ge 0\; \text{for all $c \in
    \R^n$} \qquad
  \text{(resp. $\sum
  \limits_{i,j=1}^n \beta_{ij}c_i c_j > 0$ for all $c \in \R^n
  \setminus \{0\}$);}
$$
$\bullet$ {\em copositive (resp. strictly copositive)} if
$$
\sum \limits_{i,j=1}^n \beta_{ij}c_i c_j \ge 0 \;\text{for all $c \in C^n_+$}
\qquad \text{(resp. $\sum \limits_{i,j=1}^n \beta_{ij}c_i c_j
  > 0$ for all $c \in C^n_+  \setminus \{0\}$);}
$$
We note that copositivity is a weaker condition than positive
semidefiniteness. In case $n \le 4$ every copositive matrix can
be written as a sum of a positive semidefinite matrix and a matrix having only nonnegative components, but this is not true for $n \ge 5$, see \cite{diananda:62}.
Copositive matrices play a significant role in quadratic programming (see \cite{jacobson}), while
-- up to our knowledge --
they have not been discussed in the context of
elliptic systems yet. In case
$n \le 4$, strictly copositivity can be characterized explicitly by
inequalities between the matrix coefficients, see e.g.
\cite{pingyuyu}. In particular,\\
$\bullet$ $B \in S(2)$ is strictly copositive if and only if
\begin{equation}
\label{eq:strictly copositive in dim n=2}
\beta_{11},\beta_{22}>0 \qquad \text{and}\qquad
\beta_{12}>-\sqrt{\beta_{11}\beta_{22}}.
\end{equation}
$\bullet$ $B\in S(3)$ is strictly copositive if and only if
\begin{equation}
\label{eq:strictly copositive in dim n=3}
\left \{
  \begin{aligned}
  &\beta_{11},\beta_{22},\beta_{33}>0,\\
  &\beta_{12}+\sqrt{\beta_{11}\beta_{22}}>0,\;
  \beta_{13}+\sqrt{\beta_{11}\beta_{33}}>0,\;
  \beta_{23}+\sqrt{\beta_{22}\beta_{33}}>0 \qquad \text{and}\\
  &\sqrt{\beta_{11}\beta_{22}\beta_{33}} + \beta_{12}\sqrt{\beta_{33}}+ \beta_{13}\sqrt{\beta_{22}} + \beta_{23}\sqrt{\beta_{11}}+\\
&+\sqrt{ 2 (\beta_{12}+\sqrt{\beta_{11}\beta_{22}}) (\beta_{13}+\sqrt{\beta_{11}\beta_{33}}) (\beta_{23}+\sqrt{\beta_{22}\beta_{33}})  }>0.
\end{aligned}
\right.
\end{equation}
Our first result shows that strict copositivity is a
necessary assumption for the nonexistence of nontrivial solutions of
(\ref{eq:special-case}).

\begin{theorem}
\label{sec:exist-nonex-rsn-1}
Let $N\leq 3$, suppose that $\beta_{ii} \ge 0$ for $i=1,\dots,n$, and that
the matrix $B= (\beta_{ij})_{ij}$ is {\em not} strictly copositive.
Then \eqref{eq:special-case} admits a nontrivial solution.
\end{theorem}

In fact we will prove the following stronger existence result
for the Neumann problem corresponding to (\ref{eq:special-case}) in bounded domains, which immediately gives rise to Theorem~\ref{sec:exist-nonex-rsn-1}
by tiling $\R^\sN$ with cubes and reflecting solutions.

\begin{theorem}
\label{sec:exist-nonex-neumann}
Suppose that $\Omega \subset \R^{\sN}$ is a bounded domain, $N\leq 3$, and that
the matrix $B= (\beta_{ij})_{ij} \in S(n)$ is {\em not} strictly
copositive but satisfies $\beta_{ii} \ge 0$ for $i=1,\dots,n$.
Then the Neumann problem
\begin{equation}
  \label{eq:neumann-problem}
\left \{
  \begin{aligned}
  -& \Delta u_i = \sum_{j=1}^n \beta_{ij}u_j^2 u_i,\quad\;u_1,\dots,u_n \ge 0&&\qquad
\text{in $\Omega$},\\
&\frac{\partial u_1}{\partial \nu}=\frac{\partial u_2}{\partial \nu}=\dots=\frac{\partial u_n}{\partial \nu}=0 &&\qquad \text{on $\partial \Omega$},
  \end{aligned}
\right.
\end{equation}
admits a nontrivial solution.
\end{theorem}

We note that the assumption on the nonnegativity of the diagonal
elements of $B$ is crucial in Theorem~\ref{sec:exist-nonex-rsn-1}, which already can be seen by looking at
the equation $-\Delta u=-u^3$. By a classical result \cite{walter:55}, this
equation does not admit nontrivial nonnegative
solutions defined on all of $\R^\sN$. Next, we discuss
sufficient conditions for the nonexistence of nontrivial solutions of
(\ref{eq:special-case}). It was observed in \cite{dancer.wei.weth:10}
that, in the special two component case $n=2$, the strict copositivity
of $B \in S(2)$ is also a sufficient condition. This was proved as
follows.
Assuming by contradiction that
(\ref{eq:special-case}) admits a nonnegative nontrivial solution
$u=(u_1,u_2)$, it was shown that a suitable linear combination $w=\mu_1
u_1 + \mu_2 u_2$ is a positive solution of the differential inequality
$-\Delta w \ge w^3$, contrary to a result of Gidas \cite{gidas:80}. To
exploit the idea in the more general $n$-component case, we are lead to
introduce another notion of positivity of a symmetric matrix $B$.
\begin{definition}
\label{cubic-positivity}
We call a matrix $B \in S(n)$ {\em strictly cubically copositive} if there exists
$\mu \in C^n_+$ such that
\begin{equation}
  \label{eq:2}
\sum \limits_{i,j=1}^n \beta_{ij}c_j^2c_i  \mu_i> 0 \qquad \text{for all
$c \in C^n_+ \setminus \{0\}$}.
\end{equation}
\end{definition}

We briefly comment on this definition. By applying (\ref{eq:2}) to coordinate vectors, we see that
$\mu$ must have strictly positive components to satisfy this
condition.
Moreover, by homogeneity, there is a constant $\kappa=\kappa(B,\mu)>0$ such that
\begin{equation}
  \label{eq:8}
\sum_{i,j=1}^n \beta_{ij}c_j^2 c_i \mu_i \ge \kappa \Bigl( \sum_{i=1}^n \mu_i c_i
\Bigr)^3\qquad \text{for every $c \in C^n_{+}\setminus \{0\}$.}
\end{equation}
To explain the degree of freedom given by the choice of $\mu$,
we note that if $B$ satisfies
$$
\sum \limits_{i,j=1}^n \beta_{ij}c_j^2c_i  \mu_i> 0 \qquad \text{for all
$c \in C^n_+ \setminus \{0\}$}.
$$
then the matrix $\tilde B = \diag(\mu_1^2,\dots,\mu_n^2) B
\diag(\mu_1^2,\dots,\mu_n^2)$ with components $\tilde \beta_{ij}=\mu_i^2
\beta_{ij}\mu_j^2$
satisfies
$$
\sum \limits_{i,j=1}^n \tilde \beta_{ij}c_j^2c_i> 0 \qquad \text{for all
$c \in C^n_+ \setminus \{0\}$}.
$$

Our motivation to introduce this notion is given by
the following observation.

\begin{proposition}
\label{cubically-copositive-nonex}
Let $N\leq 3$. If $B \in S(n)$ is strictly cubically copositive, then
\eqref{eq:special-case} does not admit a nontrivial solution.
\end{proposition}

Since the proof is very simple, we give it immediately.

\begin{proof}
Let $\mu \in C^n_+$ be as in the definition above.
We now suppose by contradiction that \eqref{eq:special-case} admits a nontrivial
solution $u=(u_1,\dots,u_n)$. Let $\kappa=  \kappa(B,\mu)>0$ satisfy
(\ref{eq:8}). Then the positive function $v:= \sqrt{\kappa} \sum \limits_{i=1}^n \mu_i u_i$ satisfies
$$
-\Delta v= \sqrt{\kappa} \sum_{i,j=1}^n \beta_{ij}u_j^2 \mu_i u_i \ge
\kappa^{\frac{3}{2}}\Bigl(\sum \limits_{i=1}^n \mu_i u_i\Bigr)^3 = v^3 \qquad
\text{in $\R^N$,}
$$
in contradiction, for $N\leq3$, with the aforementioned result of Gidas \cite{gidas:80}.
\end{proof}

It is natural to ask wether strict copositivity and strict cubical
copositivity are related in some way. This is answered by the
following proposition.

\begin{proposition}
\label{scc-implies-sc}
Let $B \in S(n)$.
\begin{enumerate}
\item If $B$ is strictly cubically copositive, then it is also
strictly copositive.
\item If $n=2$ and $B$ is strictly copositive, then it is also
  strictly cubically copositive.
\end{enumerate}
\end{proposition}

By combining Theorem~\ref{sec:exist-nonex-rsn-1} with Propositions
\ref{cubically-copositive-nonex} and \ref{scc-implies-sc}, we
immediately get the following

\begin{corollary}
\label{coro:complete_result 1}
If $n=2$ and $\beta_{11},\beta_{22}$ are nonnegative, then the system
(\ref{eq:special-case}) admits a nontrivial solution if and only if
$B$ is not strictly copositive, i.e., if one of the strict
inequalities
$$
\beta_{11}>0,\quad \beta_{22}>0, \quad \beta_{12}>-\sqrt{\beta_{11}
  \beta_{22}}
$$
is not satisfied.
\end{corollary}

The equivalence of strict copositivity and strict cubic copositivity
stated in Proposition~\ref{scc-implies-sc} for $n=2$ fails to be true if $n \ge 3$. Indeed, for $\eps>0$ small, the following matrix is strictly
copositive but not strictly cubically positive.
\begin{equation}
  \label{eq:4}
B_\eps= \left (
\begin{array}{ccc}
1&-1+\eps&-1+\eps\\
-1+\eps &1&1\\
-1+\eps &1&1
\end{array}
\right), \qquad \eps>0.
\end{equation}
The strict copositivity follows directly from \eqref{eq:strictly
  copositive in dim n=3}, but it is not at all obvious that
$B_\eps$ is not strictly cubically positive. We postpone the proof
of this fact to the Appendix.

In the multicomponent case $n \ge 3$, the results presented so far still leave a gap between necessary
and sufficient conditions for the nonexistence of solutions
of (\ref{eq:special-case}). Somewhat surprisingly, we can close this gap
in case $N\leq 2$ but not in the threedimensional case.

\begin{theorem}
\label{planar-case}
If $N\leq 2$, $n \ge 2$ and $B\in S(n)$ is strictly copositive, then \eqref{eq:special-case} does not admit a nontrivial solution.
\end{theorem}

By combining Theorems \ref{sec:exist-nonex-rsn-1} and \ref{planar-case} we obtain the following result.

\begin{corollary}\label{coro:complete_result 2}
Let $N\leq 2$ and let $B\in S(n)$, $n\geq 2$, be such that $\beta_{ii}\geq 0$. Then the system \eqref{eq:special-case} admits a nontrivial solution if and only if $B$ is not strictly copositive.
\end{corollary}

The proof of Theorem \ref{planar-case} relies on a test function argument which does
not extend to the three-dimensional case. Hence in the case $N \ge 3$,
$n \ge 3$ it is still open whether nonexistence of nontrivial
solutions follows from weaker assumptions than strict cubic
copositivity of $B$. We conjecture that, as in the case $N \le 2$, strict
copositivity is sufficient. Since we are not able to prove this, we add a
simple condition on the coefficients of $B$ which guarantees strict
cubic copositivity and therefore nonexistence of solutions of (\ref{eq:special-case}).

\begin{proposition}
\label{sec:introduction--1}
Suppose that
$$
\beta_{ii}> 0 \qquad \text{and}\qquad \sum_{{j=1}\atop {j
    \not=i}}^n\beta_{ij}^- >  -\beta_{ii} \qquad \text{for
  $i=1,\dots,n$,}
$$where $\beta_{ij}^-=\min\{\beta_{ij},0\}$.
Then $B$ is strictly cubically copositive, and therefore
\eqref{eq:special-case} does not admit a nontrivial
solution by Proposition~\ref{cubically-copositive-nonex}.
\end{proposition}

The paper is organized as
follows. In Section~\ref{sec:proof-exist-result-3} we will consider
the Neumann problem (\ref{eq:neumann-problem}), and we will give the
proof of Theorem~\ref{sec:exist-nonex-neumann}. The solution is found
by variational methods. More precisely, we will consider a
$\cC^1$-functional $E$ such that critical points of $E$ are precisely
(weak) solutions of (\ref{eq:special-case}).
Moreover, we will use the assumption
that $B$ is not strictly copositive to set up a suitable minimax
principle which eventually gives rise to a nontrivial critical point
of $E$. The difficulty in analyzing the functional geometry of $E$ is
the fact that zero is not a minimum but a highly degenerate critical (saddle)
point of $E$.

In Section~\ref{sec:matr-cond-nonex} we will give the proof of our
other results presented above which are concerned with matrix
properties and nonexistence of solutions of (\ref{eq:special-case}).
Afterwards, in Section \ref{sec:generalization} we will add some extensions of
our results to the more general system
\begin{equation}
  \label{generalization}
-\Delta u_i = \sum_{j=1}^n \beta_{ij}u_j^\frac{p}{2} u_i^{\frac{p}{2}-1},\qquad
u_1,\dots,u_n \ge 0 \qquad \text{in $\R^\sN$,}
\end{equation}
where $2<p<2^*=2N/(N-2)$ if $N\geq 3$ and $2<p<\infty$ if $N \in
\{1,2\}$.

Finally, in the Appendix we will give the proof that the matrix $B_\varepsilon$ defined in \eqref{eq:4} is not strictly cubically positive for sufficiently small $\varepsilon>0$.

\section{Proof of the existence result}
\label{sec:proof-exist-result-3}

This section is devoted to the proof of Theorem
\ref{sec:exist-nonex-neumann}. From now on we assume that the
matrix $B= (\beta_{ij})_{ij}$ is {\em not} strictly copositive, but
$\beta_{ii} \ge 0$ for $i=1,\dots,n$. We wish to show that, in this
case, (\ref{eq:neumann-problem}) admits a nontrivial solution. Without loss of generality, we may from now on assume that
\begin{equation}
  \label{eq:without-loss-1}
B c \not= 0\qquad \text{for every $c \in
C^n_+ \setminus \{0\}$.}
\end{equation}
Indeed, if there is $c \in
C^n_+ \setminus \{0\}$ with $Bc=0$, then the constant vector
\mbox{$u \equiv (\sqrt{c_1},\dots,\sqrt{c_n})$} is a nontrivial
solution of \eqref{eq:special-case}, and the assertion holds.\\
Moreover, we may also assume that
\begin{equation}
  \label{eq:beta-ii-positive}
\beta_{ii}>0\qquad \text{for $i=1,\dots,n$,}
\end{equation}
otherwise the $i$-th coordinate vector $e_i$ is a constant nontrivial
solution of \eqref{eq:neumann-problem}. Next, we consider
$$
\partial C^n_+ := \{x \in C^n_+ \::\: x_i= 0\quad \text{for some
  $i$} \}.
$$
Arguing by induction on $n$, we may from now on assume that
\begin{equation}
  \label{eq:boundary-positive}
\sum_{i,j=1}^n \beta_{ij} c_i c_j >0 \qquad \text{for all $c \in \partial C^n_+  \setminus \{0\}$}.
\end{equation}
Indeed, if $n=2$, then \eqref{eq:beta-ii-positive} assures that \eqref{eq:boundary-positive} holds. On the other hand,
if $n \ge 3$ and $\sum \limits_{i,j=1}^n \beta_{ij} c_i c_j \le 0$ for some $c \in
\partial C^n_+$ with, say, $c_k=0$, then we may eliminate the $k$--th column and the $k$--th row from $B$ and obtain
a matrix $\tilde B= (\tilde \beta_{ij})_{ij} \in \R^{{(n-1)} \times
  {(n-1)}}$ which is not strictly copositive. By induction, we then
get a nontrivial solution $v: \Omega \to \R^{n-1}$ of the reduced system
\begin{equation}
  \label{eq:5}
\left \{
  \begin{aligned}
  -\Delta v_i&= \sum_{j=1}^{n-1}\tilde \beta_{ij} v_j^2
v_i,\quad v_1,
\dots, v_{n-1} \ge 0 &&\qquad \text{in $\Omega$,}\\
\frac{\partial v_1}{\partial \nu}&=\frac{\partial v_2}{\partial \nu}=\dots=\frac{\partial v_{n-1}}{\partial \nu}=0 &&\qquad \text{on $\partial \Omega$},
  \end{aligned}
\right.
\end{equation}
Then a nontrivial solution of the original problem (\ref{eq:neumann-problem}) is given by
$$
u: \Omega \to \R^n,\qquad u=(v_1,\dots,v_{k-1},0,v_k,\dots,v_{n-1}).
$$

We need to introduce some more notation. We consider the Hilbert space $\cH:= H^1(\Omega,\R^n)$, endowed with
the norm
$$
\|u\|^2:= \int_\Omega(|\nabla u|^2 +u^2)\:dx \qquad \text{for $u=(u_1,\dots,u_n) \in \cH$}.
$$
Here and in the following we use the notation
$$
|u|^2= \sum_{i=1}^n
u_i^2, \qquad
|u^-|^2= \sum_{i=1}^n (u_i^-)^2 \qquad \text{and}\qquad |\nabla u|^2= \sum_{i=1}^n |\nabla u_i|^2.
$$

\begin{lemma}
\label{sec:proof-exist-result}
Consider
$$
E: \cH \to \R, \qquad E(u)= \frac{1}{2}\int_\Omega \bigl(|\nabla
u|^2+|u^-|^2
\bigr)\:dx -\phi(u),
$$
where
$$
\phi(u):=\frac{1}{4} \sum_{i,j=1}^n \int_\Omega \beta_{ij}
(u_i^+)^2 (u_j^+)^2  \:dx \qquad \text{for
  $u \in \cH$}
$$
Then we have:
\begin{enumerate}
\item[(i)] $E$ is a $\cC^1$-functional, and critical points of $E$
are nonnegative solutions of \eqref{eq:neumann-problem}.
\item[(ii)] $E$ satisfies the Palais-Smale condition.
\end{enumerate}
\end{lemma}

\begin{proof}
i) The fact that $E$ is of class $\cC^1$ follows from standard arguments
in the calculus of variations, using the Sobolev embeddings $\cH
\hookrightarrow L^2(\Omega)$ and $\cH
\hookrightarrow L^4(\Omega)$.  If $u \in \cH$ is a critical point of
$E$, then $u$ satisfies
$$
\int_\Omega (\nabla u_i \cdot \nabla \phi+ u_i^- \phi)\,dx = \sum_{j=1}^n
\int_\Omega  \beta_{ij} (u_j^+)^2 u_i^+ \phi\,dx
$$
for $\phi \in H^1(\Omega)$ and $i=1,\dots,n$. By choosing $\phi= u_i^-$, we obtain
$$
\int_\Omega (|\nabla u_i^-|^2+|u_i^-|^2)\:dx = 0 \qquad \text{for $i=1,\dots,n$.}
$$
This implies that $u_i^-\equiv 0$ for $i=1,\dots,n$, and hence
$u=(u_1,\dots,u_n)$ is a solution of \eqref{eq:neumann-problem}.\\
ii) Let $(u_k)_k \subset \cH$ be a sequence such that
$E(u_k)$ remains bounded and $E'(u_k) \to 0$ in the dual space $\cH'$.
Then
\begin{align*}
{\rm o}(\|u_k\|)&= E'(u_k)u_k =
\int_{\Omega} \bigl(|\nabla u_k|^2+|u_k^-|^2\bigr)dx -\sum_{i,j=1}^n \int_\Omega \beta_{ij} ((u_k^i)^+)^2 ((u_k^j)^+)^2 \:dx\\
&=4E(u_k)-\int_{\Omega}\bigl(|\nabla u_k|^2+|u_k^-|^2\bigr)dx,
\end{align*}
and hence
\begin{equation}\label{eq:int-of-|nabla-un|^2-is-bounded}
\int_{\Omega}\bigl(|\nabla u_k|^2+|u_k^-|^2\bigr)dx \le C +{\rm o}(\|u_k\|) \qquad \text{ as }k\to +\infty,
\end{equation}
for some constant $C>0$. We now suppose by contradiction that
$(u_k)_k$ is unbounded in $\cH$, hence $\|u_k\| \to \infty$ up to a
subsequence. Define $v_k:= \frac{u_k}{\|u_k\|}$. Then $\|v_k\|=1$ for
all $k$ and $\lim \limits_{k \to \infty}\int_{\Omega}\bigl(|\nabla
v_k|^2+(v_k^-)^2\bigr)dx = 0$, so we may pass to a subsequence such
that $v_k \to v$ in $\cH$, where $v \not=0$ is a constant vector with $v_i \ge 0$ for
$i=1,\dots,n$. Moreover, for arbitrary $\phi \in H^1(\Omega)$ and
$i=1,\dots,n$ we have
\begin{multline*}
\sum_{j=1}^n \beta_{ij} v_j^2 v_i \int_{\Omega}\phi\:dx = \sum_{j=1}^n \int_\Omega \beta_{ij} (v_j^+)^2(v_i^+) \varphi\, dx=\\
= \lim_{k \to \infty} \|u_k\|^{-3}\sum_{j=1}^n \int_{\Omega} \beta_{ij} ((u_k^j)^+)^2 (u_k^i)^+  \phi\:dx =\\
=\lim_{k \to \infty} \|u_k\|^{-3}\Bigl(\partial_i
E(u_k)\phi-\int_{\Omega} \bigl(\nabla u_k^i \cdot \nabla \phi+ (u_k^i)^-
\,\phi\bigr)
\:dx \Bigr)= 0.
\end{multline*}
Consequently,
$$
\sum_{j=1}^n \beta_{ij} v_j^2 v_i = 0  \qquad \text{for
  $i=1,\dots,n$.}
$$
If $v \in \partial C^n_+ \setminus \{0\}$, this obviously contradicts
\eqref{eq:boundary-positive}. On the other hand, if $v \in
\innt(C^n_+)$, then $\sum_{j=1}^n \beta_{ij} v_j^2 = 0$ for
 $i=1,\dots,n$, contradicting \eqref{eq:without-loss-1}.\\
We therefore conclude that $\|u_n\|$ is bounded. Next, we note that
$\nabla E(u_n)= u_n -A(u_n)$ with
$$
A: \cH \to \cH,\qquad  A w = (-\Delta + 1)^{-1}
\left (
\begin{array}{c}
w_1^+ +\sum \limits_{j=1}^n \beta_{1j} (w_j^+)^2 w_1^+  \\
\vdots\\
w_n^+ +\sum \limits_{j=1}^n \beta_{nj} (w_j^+)^2 w_n^+
\end{array}
\right),
$$
i.e., the $i$-th component $(Aw)_i$ of $Aw$ is uniquely given by
$$
\int_{\Omega}\Bigl( \nabla (Aw)_i \cdot \nabla \phi + (Aw)_i\, \phi\Bigr)\:dx
= \int_{\Omega} \Big( w_i^+ \phi +\sum \limits_{j=1}^n \beta_{ij} (w_j^+)^2 w_i^+ \phi  \Big)
\:dx,
$$
for all $\phi \in H^1(\Omega)$. By the compactness
of the embeddings $\cH \hookrightarrow L^3(\Omega,\R^n)$ and $\cH \hookrightarrow
L^1(\Omega,\R^n)$, we see that $A$ is also a compact operator. Hence
we may pass to a subsequence of $(u_k)_k$ such that $A(u_k) \to \bar u$ in $\cH$. But then
also
$$
\lim_{k \to \infty}(u_k-\bar u)= \lim_{k \to
  \infty}(u_k-A(u_k))=\lim_{k \to \infty} \nabla E(u_k)= 0.
$$
Hence $u_k \to \bar u$ strongly in $\cH$, which was claimed.
\end{proof}

In order to prove Theorem \ref{sec:exist-nonex-neumann}, our goal is
to set up a minimax principle which gives rise to a positive
critical value of $E$. For this we need some preparations. We let $b: \R^n
\to \R$ denote the quadratic form associated with $B$, i.e.,
\begin{equation}
  \label{eq:9}
b(c)= \sum \limits_{i,j=1}^n \beta_{ij}\, c_i c_j \qquad \text{for $c
  \in \R^n$.}
\end{equation}
By the assumption that $B$ is not strictly copositive and by (\ref{eq:boundary-positive}), there exists $d:=(d_1,\dots,d_k) \in \innt(C^n_+)$ such that
$b(d) \le 0$. In fact, we can find $d \in \innt(C^n_+)$ such that
$b(d)<0$, since otherwise $\min \limits_{C^n_+}b=0$ would be
attained at a point $c \in \innt(C^n_+)$ satisfying
$0=\nabla b(c)= 2 B c$, contradicting \eqref{eq:without-loss-1}.\\
From now on we fix $d \in \innt(C^n_+)$ such that $b(d^2)<0$, where $d^2:=(d_1^2,\ldots,d_n^2)$. We define the linear map
\begin{equation}
  \label{eq:10}
\cL: \cH \to \R^n, \qquad  \cL u = \Bigl(\int_\Omega u_1
\,dx,\dots,\int_\Omega u_n \,dx \Bigr)
\end{equation}
and, for $\lambda>0$, the sets
$$
M_\lambda := \{u \in \cH\::\: \|u\|=\lambda,\: \cL u \in \R d\}.
$$

\begin{lemma}
  \label{sec:proof-exist-result-1}
There exists $\lambda>0$ such that
\begin{equation}
  \label{eq:estimate-linking-final}
\sigma_\lambda:= \inf_{M_\lambda}E>0.
\end{equation}
\end{lemma}

\begin{proof}
We first show that there exists $\kappa_1>0$ such that
\begin{equation}
  \label{eq:estimate-linking}
\max \Bigl\{-\phi(u),\Bigl(\int_\Omega \bigl(|\nabla
u|^2+|u^-|^2\bigr)\:dx\Bigr)^2 \Bigr\} \ge \kappa_1 \qquad \text{for all $u \in
  M_1$}.
\end{equation}
Indeed, suppose by contradiction that there exists a sequence $(u_k)_k
\subset M_1$ such that
$$
\liminf_{k \to \infty}\phi(u_k) \ge 0 \qquad \text{and}\qquad \lim_{k
  \to \infty} \int_\Omega \bigl(|\nabla
u_k|^2+|u_k^-|^2\bigr)\:dx = 0.
$$
Since $\|u_k\|= 1$ for all $k$, we may pass to a
subsequence such that $u_k \to u$ in $\cH$ as
$k \to \infty$, where $u \not=0$ is a constant vector with $u_i \ge 0$ for
$i=1,\dots,n$. By continuity of the map $\cL$ and the functional
$\phi$, we find that
\begin{equation}
  \label{eq:contradiction}
\cL u \in \R d \qquad \text{and} \qquad \phi(u) \ge 0.
\end{equation}
Since $u$ is a constant vector, we conclude that $u \equiv \lambda d$ for some $\lambda > 0$. By the
choice of $d$ we deduce that $\phi(u)<0$, contrary to \eqref{eq:contradiction}.
 Thus we have proved
\eqref{eq:estimate-linking}.
By homogeneity, we deduce that, for every $\lambda>0$,
\begin{equation*}
\max\{-\phi(u),\Bigl(\int_\Omega |\nabla u|^2 +|u^-|^2\bigr) \:dx\Bigr)^2\} \ge \kappa_1
\lambda^4 \qquad \text{for all $u \in M_\lambda$.}
\end{equation*}
On the other hand, it is also clear that there exists $\kappa_2>0$
independent of $\lambda>0$ such that
\begin{equation*}
\phi(u) \le \kappa_2 \|u\|^4 = \kappa_2
\lambda^4 \qquad \text{for all $u \in M_\lambda$.}
\end{equation*}
We now claim that \eqref{eq:estimate-linking-final} holds for $\lambda=
\frac{\sqrt[4]{\kappa_1}}{2\sqrt{\kappa_2}}$. Indeed, let $u \in M_\lambda$.
Then
$$
-\phi(u) \ge \kappa_1
\lambda^4 \qquad \text{or}\qquad \int_\Omega \bigl(|\nabla u|^2 +|u^-|^2\bigr) \:dx \ge \sqrt{\kappa_1}
\lambda^2.
$$
In the first case we have
$$
E(u)=\int_\Omega \bigl(|\nabla u|^2 +|u^-|^2\bigr)\:dx -
\phi(u) \geq \kappa_1 \lambda^4>0,
$$
whereas in the second case
$$
E(u)\geq \frac{\sqrt{\kappa_1}\lambda^2}{2}-\kappa_2 \lambda^4 =  \lambda^2 (\frac{\sqrt{\kappa_1}}{2}  -\kappa_2 \lambda^2)=\frac{\sqrt{\kappa_1}}{4}\lambda^2 >0.
$$
We thus have established \eqref{eq:estimate-linking-final}.
\end{proof}

From now on we fix $\lambda>0$ such that (\ref{eq:estimate-linking-final}) holds, and we fix functions $\phi_i \in C^2(\overline{\Omega})$, $i=1,\dots,n$ such
that $0\leq \phi_i \le 1,\: \phi_i \not \equiv 0$ for
  $i=1,\dots,n$ and $\phi_i\, \phi_j \equiv 0$ for $i \not=j.$ We also put
$$
\kappa:= \min_{1 \le i \le n} \int_\Omega
\phi_i^2(x)\:dx>0.
$$
We define the map
\begin{equation}
\label{eq:crucial-definition}
h: C^n_+ \to \cH, \qquad c \mapsto h_c = (c_1 \phi_1,\dots, c_n \phi_n).
\end{equation}
Then we have
$$
\sum_{i,j=1}^n \int_\Omega \beta_{ij}(h_c^i)^2(h_c^j)^2\:dx= \sum_{i=1}^n \beta_{ii}\,c_i^4 \int_\Omega \phi_i^4(x)\:dx>0 \quad \text{ for every }c\in C^n_+\setminus \{0\},
$$
which by homogeneity implies that there exists $\kappa_3>0$ such that
$$
\sum_{i,j=1}^n  \int_{\Omega} \beta_{ij} (h_c^i)^2 (h_c^j)^2\:dx >
\kappa_3  |c|^4 \qquad \text{for every
  $c \in C^n_+\setminus \{0\}$.}
$$
As a consequence, there exists $R_1>0$ such that
\begin{equation}
  \label{eq:11}
E(h_c)\leq \frac{|c|^2}{2} \Bigl(\max_{1 \le i \le n} \int_\Omega |\nabla \phi_i|^2\Bigr) - \frac{\kappa_3}{4}|c|^4 \leq 0 \qquad \text{for $c \in C^n_+$ with $|c| \ge R_1$.}
\end{equation}
Next we consider the homotopy
$$
C^n_+ \times [0,1] \to \cH,\qquad
(c,t) \mapsto h_{c,t} \quad \text{with $h_{c,t}(x)=
(1-t)c + t h_c(x)$.}
$$
We note that $h_{c,t}(C^n_+) \subset \partial C^n_+$ for $c\in \partial C^n_+,\: 0 \le t \le 1$, and therefore \eqref{eq:boundary-positive} implies that
$$
\sum_{i,j=1}^n \int_\Omega \beta_{ij}(h_{c,t}^i)^2(h_{c,t}^j)^2\:dx>0
\qquad \text{ for every }c\in \partial C^n_+\setminus \{0\}, \ t\in[0,1].
$$
By reasoning exactly as in (\ref{eq:11}) we deduce the existence of $R_2>0$ such that
\begin{equation}
  \label{eq:12}
E(h_{c,t})\leq 0 \qquad \text{ for every $c\in \partial C^n_+ \setminus \{0\}$ such that $|c|\geq R_2$},\ t\in [0,1].
\end{equation}
Moreover, recalling that $0\leq \phi_i\leq 1$, we see that
$$
\|h_{c,t}\|^2 \geq \sum_{i=1}^n c_i^2 \int_\Omega
\bigl[(1-t) + t \phi_i(x)\bigr]^2\,dx
\ge \sum_{i=1}^n c_i^2 \int_\Omega  \phi_i^2(x)\,dx
\ge \kappa |c|^2
$$
for $0 \le t \le 1$, $c \in C^n_+$ and therefore
\begin{equation}
  \label{eq:7}
\|h_{c,t}\| > \lambda \qquad \text{for $0 \le t \le 1$ and $c \in C^n_+$
  with $|c| \ge R_3:=\frac{\lambda+1}{\sqrt{\kappa}}$.}
\end{equation}
Finally we take $D:=B_{R+1}(0) \cap C^n_+$ with
$R:=\max\{R_1,R_2,R_3\}$. We define the continuous function
$$
\Theta:D\to \cH,\qquad \left \{
  \begin{aligned}
&\Theta(c)=c &&\qquad \text{if $c\in B_R(0)\cap C^n_+$},\\
&\Theta(c)= h_{c,|c|-R} &&\qquad \text{if $c \in \bigl(B_{R+1}(0)\setminus B_R(0)\bigr)\cap
C^n_+$},\\
&\Theta(c)=h_c &&\qquad \text{if
$c\in C^n_+\setminus B_{R+1}(0)$,}
  \end{aligned}
\right.
$$
By combining \eqref{eq:boundary-positive} with \eqref{eq:11} and \eqref{eq:12} we see that
\begin{equation}\label{eq:Theta-leq-0-ontheboundary}
E(\Theta (c))\leq 0 \qquad \text{ for every } c\in \partial D.
\end{equation}
We are now in a position to define a minimax value for $E$.

\begin{proposition}
\label{sec:proof-exist-result-2}
Let
$$
\cT:=\{ \gamma:D\to \cH :\ \text{$\gamma$ continuous,
  $\gamma|_{\partial D}=\Theta$}\}
\qquad \text{and}\qquad
\sigma: =\inf_{\gamma\in \cT} \sup_{c \in D}E(\gamma(c)).
$$
Then $\sigma \geq \sigma_\lambda>0$, and $\sigma$ is a
critical
value of $E$.
\end{proposition}

\begin{proof}
We first show that
\begin{equation}
\label{eq:c>c-lambda_essential_nonempty_intersection}
\gamma(D)\cap M_\lambda\neq \emptyset \qquad \text{ for any $\gamma\in \cT$},
\end{equation}
then Lemma \ref{sec:proof-exist-result-1} immediately yields $\sigma \geq \sigma_\lambda>0$.
To prove the intersection property
(\ref{eq:c>c-lambda_essential_nonempty_intersection}), we will use
classical degree theory (see e.g.
\cite[Appendix D]{willem}). For this we define
$$
\cK: \cH \to \R^n,\qquad \cK u = \|u\|^2 d + P (\cL u)
$$
where $\cL$ is defined in (\ref{eq:10}) and $P: \R^n \to \R^n$ is the
orthogonal projection onto $d^\perp:= \{c \in \R^n \::\: c \cdot e = 0\}$.
We observe that \eqref{eq:c>c-lambda_essential_nonempty_intersection}
holds if and only if
\begin{equation}
  \label{eq:6}
\lambda^2 d \:\in\: [\cK \circ \gamma](D) \qquad \text{ for any $\gamma\in \cT$.}
\end{equation}
We first consider $\gamma= \Theta$. As a consequence of (\ref{eq:7})
and the definition of $\Theta$, we have for $c \in
\overline D$
$$
\cK(\Theta(c))= \lambda^2 d \qquad \text{if and only if}\qquad c= \mu d
\; \text{with $\mu=\frac{\lambda}{\sqrt{|\Omega|}|d|}$}
$$
Moreover $\mu d\in B_R(0)\cap C^n_+$ and hence, for $c$ in a neighborhood of $\mu d$ in $\R^n$ we have
$\cK(\Theta(c))= |\Omega|\Bigl( |c|^2 d +P c\Bigr)$, so that the derivative of $\cK \circ \Theta$ at $\mu d$ is given by
$$
[\cK \circ \Theta]'(\mu d)e =  |\Omega|\Bigl(2\mu [d \cdot e] d + P e\Bigr) \qquad
\text{for $e \in \R^n$.}
$$
If we choose a basis of $\R^n$ of the type $\{d,e_1,\ldots,e_{n-1}\}$, where $\{e_1,\ldots,e_{n-1}\}$ is a basis of the subspace $d^\perp$, then the matrix of the linear map $[\cK \circ \Theta]'(\mu d)$ in such basis is given by ${\rm diag}(2\mu |d|^2 |\Omega|, |\Omega|,\ldots,|\Omega|)$. Hence the Jacobian determinant of $\cK \circ \Theta$ at $\mu d$ is
$2 |\Omega|^n \mu |d|^2 >0$
and therefore $\deg (\cK \circ \Theta,D,\lambda d)=1.$
Consequently, we also have $\deg (\cK \circ \gamma,D,\lambda d)=1$ for
every $\gamma \in \cT$ by standard properties of the degree, since $\gamma \equiv
\Theta$ on
$\partial D$. Hence (\ref{eq:6}) and therefore
(\ref{eq:c>c-lambda_essential_nonempty_intersection}) holds.\\
We still need to prove that $\sigma$ is a critical value of $E$. We argue
by contradiction and assume that this is not the case. Then, since
$E$ satisfies the Palais-Smale condition, there exists $\eps \in
(0,\frac{\sigma}{2})$ such that
$$
\|\nabla E(u)\| \ge \eps \qquad \text{for all $u \in
  \cH$ with $\sigma-2 \eps \le E(u) \le \sigma+2 \eps.$}
$$
Now the quantitative
deformation lemma (see \cite[Lemma 2.3]{willem}) yields a continuous map $\eta: \cH \to \cH$ such that
$$
E(\eta(u))\leq \sigma-\eps \qquad \text{ whenever }E(u)\leq
\sigma+\eps
$$
and
$$\eta(u)=u \qquad \text{ whenever } E(u) \le \sigma- 2\eps
$$
Now let $\gamma\in \cT$ with $\sup \limits_{c \in D}E(\gamma(c)) \leq
\sigma+\eps$. Since $E \circ \gamma = E \circ \Theta \leq 0$ on
$\partial D$, we infer that $\eta \circ \gamma=\gamma= \Theta$ on $\partial D$
and therefore $\eta \circ \gamma \in \cT$. This yields
$$
c\leq \sup_{\eta\circ \gamma(D)}E \leq \sigma-\eps,
$$
a contradiction. We conclude that $\sigma$ is a critical value of $E$,
as claimed.
\end{proof}

\begin{altproof}{Theorem~\ref{sec:exist-nonex-neumann} (completed)}
By Proposition \ref{sec:proof-exist-result-2}, there exists a nontrivial critical
point of $E$, which by Lemma \ref{sec:proof-exist-result}-(i) is a solution of (\ref{eq:neumann-problem}).
\end{altproof}

\section{Nonexistence results and matrix conditions}
\label{sec:matr-cond-nonex}
In this Section we will give the proof of Theorem \ref{planar-case}
and Propositions \ref{scc-implies-sc} and
\ref{sec:introduction--1}.
We start with the

\begin{altproof}{Theorem \ref{planar-case}}
Suppose by contradiction that \eqref{eq:special-case} admits a nontrivial solution. Without loss of generality, we may assume that
$u_i>0$ in $\R^N$ for $i=1,\dots,n$. For $R>0$, consider the function
$$
f_R:\R\to \R,\qquad \left\{
\begin{aligned}
f_R(r)&=1 &&\qquad \text{if $r\leq R$,}\\
f_R(r)&=\frac{\log(r/R^2)}{\log(1/R)}&&\qquad \text{if $R\leq r\leq
  R^2$,}\\
f_R(r)&=0 &&\qquad \text{if $r\geq R^2$.}
\end{aligned}
\right.
$$
For $N=1,2$, if we take the radial function $\phi_R \in H^1(\R^N),\: \phi_R(x)=f_R(|x|)$,
we then have
$$
\int_{\R^N}|\nabla \phi_R|^2\:dx\to 0 \qquad  \text{as $R \to \infty$.}
$$
In fact, for $N=1$,
$$
\int_{\R}|\nabla \phi_R|^2\:dx=  \frac{2}{\log^2R}\int_R^{R^2}
\frac{1}{r^2} \, dr = 2 \left(\frac{1}{R}-\frac{1}{R^2}\right) \frac{1}{\log^2 R} \to 0 \qquad \text{as $R \to \infty$,}
$$
whereas for $N=2$,
$$
\int_{\R^2}|\nabla \phi_R|^2\:dx=  \frac{2\pi}{\log^2R}\int_R^{R^2}
\frac{1}{r} \, dr = \frac{2\pi}{\log R} \to 0 \qquad \text{as $R \to \infty$.}
$$
Now, multiplying (\ref{eq:special-case}) with
$\frac{\phi_R^2}{u_i}$ and integrating by parts, we get
\begin{align*}
\sum_{j=1}^n\beta_{ij}\int_{\R^N} u_j^2 \phi_R^2\:dx= \int_{\R^N} \frac{-\Delta u_i}{u_i}\phi_R^2\:dx=
\int_{\R^N} \nabla u_i \cdot \left(\frac{2 \phi_R \nabla \phi_R}{u_i}- \phi_R^2 \frac{\nabla u_i}{u_i^2}
\right)\:dx\\
=- \int_{\R^N} \left|\frac{\phi_R}{u_i} \nabla u_i-\nabla \phi_R \right|^2\: dx +  \int_{\R^N} |\nabla \phi_R|^2\:dx
\le \int_{\R^N} |\nabla \phi_R|^2\:dx={\rm o}(1)
\end{align*}
as $R \to \infty$. Next we let $c_j(R):= \int_{\R^N} u_j^2 \phi_R^2\:dx$ for $j=1,\dots,n$. By multiplying the above inequality with $c_i(R)$ and
summing over $i$, we obtain from the strict copositivity of the matrix $B$
$$
0 \le {\rm const} \sum_{i=1}^n c_i^2(R) \le \sum_{i,j=1}^n \beta_{ij}c_i(R) c_j(R) \le {\rm o}(1)  \sum_{i=1}^n c_i(R) \qquad \text{as $R \to \infty$},
$$
Thus $\int_{B_R(0)} u_i^2 \leq {\rm o}(1)\to 0$ as $R\to +\infty$ and hence $u_i\equiv 0$ for every $i$, contrary to what we have assumed.
\end{altproof}

\begin{altproof}{Proposition \ref{scc-implies-sc}}
First we show i), so we assume that $B \in S(n)$ is strictly cubically
copositive. Hence there exists $\mu \in C^n_+$ such that $\sum
\limits_{i,j=1}^n \beta_{ij}c_i^2c_j  \mu_j> 0$ for every $c\in C^n_+ \setminus \{0\}$.
To show strict copositivity of $B$, we need to prove that $b(c)>0$
for  $c\in C^n_+ \setminus \{0\}$, where $b: \R^n \to \R$ denotes the quadratic form associated with
$B$ (see (\ref{eq:9})) or, equivalently, that $\tilde b(c)=b(c^2)=b(c_1^2,\ldots, c_n^2)>0$ for $c\in C^n_+\setminus \{0\}$. For a nonempty subset $\cN \subset
\{1,\dots,n\}$, we put
$$
C_\cN:= \{c \in C^n_+\setminus \{0\}\::\: c_i=0\: \text{ for $i \not \in
  \cN$}\}.
$$
Arguing by induction on $|\cN|$, we prove that, for every $\cN \subset \{1,\dots,n\}$,
\begin{equation}
  \label{eq:induction-cubically-copositivity}
\tilde b(c)>0 \qquad \text{for all $c \in C_\cN$.}
\end{equation}
If $|\cN|=1$, then $\cN=\{i\}$ for some $i=1,\dots,n$, and choosing  $c=e_i$ in
Definition~\ref{cubic-positivity} immediately gives $\beta_{ii}>0$ and
therefore \eqref{eq:induction-cubically-copositivity}.\\
Next we fix $l \in \{2,\dots,k\}$, and we suppose that
\eqref{eq:induction-cubically-copositivity} holds for all $\cN$ with
$|\cN|\le l-1$. For each ${\cN_*}$ with $|{\cN_*}|= l$, we
consider
$\tilde \mu = (\tilde \mu_1,\dots,\tilde \mu_n) \in C^n_+$,
where $\tilde \mu_i=\mu_i$ if $i \in \cN_*$ and $\tilde \mu_i=0$ if $i \notin \cN_*$.
We note that for every $c \in C_{{\cN_*}}$ we have
$$
\partial_{\tilde \mu} \tilde b(c)= \nabla \left(\sum \limits_{i,j =1}^n
\beta_{ij}c_i^2c_j^2\right)\cdot \tilde \mu =
4 \sum \limits_{i,j \in \cN_*} \beta_{ij}c_j^2c_i  \tilde \mu_i= 4 \sum \limits_{i,j=1}^n \beta_{ij}c_i^2c_j  \mu_j> 0
$$
As a consequence, by integrating the previous expression we deduce that
$\tilde b(c)>0$ for all $c$ that can be written as $c= \hat c + t \tilde \mu$
with $\hat c= 0$ or $\hat c \in C_\cN$ for some $\cN$ with $|\cN|\le
l-1$ and $t>0$. Since every $c \in C_{{\cN_*}}$ can be written in
this way, we conclude that \eqref{eq:induction-cubically-copositivity}
holds for every element of $C_{{\cN_*}}$.\\
Next we prove ii), arguing somewhat more directly than in \cite[Theorem 2.1]{dancer.wei.weth:10}. Let$$
B= \left (
  \begin{array}{cc}
  \beta_{11}& \beta_{12}\\
  \beta_{12}& \beta_{22}
  \end{array}
\right)
\quad \in \quad S(2)
$$
be strictly copositive, so that $\beta_{11},
\beta_{22}>0$ and $\beta_{12}> -\sqrt{\beta_{11} \beta_{22}}$ by (\ref{eq:strictly copositive in dim n=2}). To show the
strict cubic copositivity of $B$, we consider the vector
$\mu:=(\frac{1}{\sqrt[4]{\beta_{11}}},\frac{1}{\sqrt[4]{\beta_{11}}})
        \in C^2_+$. Take an arbitrary $c=(c_1,c_2) \in C^2_+\setminus \{(0,0)\}$.
        If either $c_1=0$ or $c_2=0$  then  $\sum_{i,j=1}^2 \beta_{ij} c_j^2 \mu_i c_i$ is either equal to $\beta_{22}c_2^3\mu_2>0$ or $\beta_{11}c_1^3\mu_1>0$ respectively. Suppose now that $c_1,c_2\neq 0$ and put $\tilde c_i=c_i \sqrt[4]{\beta_{ii}}$
          for $i=1,2$. Then
\begin{align*}
&\sum_{i,j=1}^2 \beta_{ij} c_j^2 \mu_i c_i=
\frac{1}{\sqrt[4]{\beta_{11}}}\ (\beta_{11} c_1^3+ \beta_{12}
    c_2^2 c_1) + \frac{1}{\sqrt[4]{\beta_{22}}} (\beta_{22} c_2^3+ \beta_{12}
    c_1^2 c_2)\\
&=\tilde c_1^3 +\tilde c_2^3 + \frac{\beta_{12}}{\sqrt{\beta_{11}\beta_{22}}} (\tilde c_2^2 \tilde c_1 + \tilde c_1^2 \tilde
  c_2)
>\tilde c_1^3 +\tilde c_2^3 -(\tilde c_2^2 \tilde c_1 + \tilde c_1^2 \tilde
  c_2)\geq 0,
\end{align*}
as required.
\end{altproof}

\begin{altproof}{Proposition~\ref{sec:introduction--1}}
Using the simple inequality $s^2 t +s t^2 \le s^3+t^3$ for $s,t
\ge 0$ and the fact that $B=(\beta_{ij})_{ij}$ is symmetric, we obtain
\begin{align*}
\sum_{i,j=1}^n \beta_{ij}c_j^2 c_i &\ge \sum_{i=1}^n \beta_{ii}c_i^3
+ \frac{1}{2}\sum_{{i,j=1}\atop{i\not= j}}^n \beta_{ij}^-(c_j^2
c_i+c_i^2 c_j) \ge \sum_{i=1}^n \beta_{ii}c_i^3
+ \frac{1}{2}\sum_{{i,j=1}\atop{i\not= j}}^n \beta_{ij}^-(c_i^3+c_j^3)\\
&= \sum_{i=1}^n \beta_{ii}c_i^3
+\sum_{{i,j=1}\atop{i\not= j}}^n \beta_{ij}^-c_i^3= \sum_{i=1}^n c_i^3 \Bigl(\beta_{ii}
+\sum_{{j=1}\atop{j\not=i}}^n \beta_{ij}^-\Bigr)\qquad \text{for $c \in C^n_{+}$.}
\end{align*}
Moreover, since
$$
\kappa_0:= \min_{i=1,\dots,n}\Bigl(\beta_{ii}+ \sum_{{j=1}\atop{j
    \not=i}}^n\beta_{ij}^-\Bigr) >0
$$
we infer that $\sum \limits_{i,j=1}^n \beta_{ij}c_j^2 c_i \ge \kappa_0 \sum
\limits_{i=1}^n
c_i^3 >0$ for every $c \in C^n_{+}\setminus\{0\}$.
\end{altproof}

\section{Results for systems with more general power-type nonlinearities}
\label{sec:generalization}

Some of the results that we have obtained for the cubic system \eqref{eq:special-case} can be extended to more general systems such as
\begin{equation}
  \label{eq:general-case}
-\Delta u_i = \sum_{j=1}^n \beta_{ij}u_j^\frac{p}{2} u_i^{\frac{p}{2}-1},\qquad
u_1,\dots,u_n \ge 0 \qquad \text{in $\R^\sN$,}
\end{equation}
where now the dimension $N$ is arbitrary and $2<p<2^*=2N/(N-2)$ if $N\geq 3$ and $2<p<\infty$ if $N \in
\{1,2\}$ {\color{red} (in fact, in this setting, the picture is less clear, and at the moment we need some further restriction on $p$, see below)}. In this section we state and prove such extensions. Observe that \eqref{eq:general-case} reduces to \eqref{eq:special-case} when $p=4$.

Most of the techniques used in the proofs will be simple adaptations of the ones used in the previous two sections. In such cases, we will only provide a sketch of the proof, stressing the major differences with respect to the cubic case.

Concerning the existence of nontrivial solutions of \eqref{eq:general-case}, we have the following.

\begin{theorem}
\label{sec:exist-nonex-rsn-1_withp}
Suppose that $2<p<2^*=2N/(N-2)$ if $N\geq 3$ and $2<p<\infty$ if $N \in
\{1,2\}$. Suppose furthermore that $\beta_{ii} \ge 0$ for $i=1,\dots,n$, and that
the matrix $B= (\beta_{ij})_{ij}\in S(n)$ is {\em not} strictly copositive.
Then \eqref{eq:general-case} admits a nontrivial solution.
\end{theorem}

As already discussed in the special case $p=4$, this result is an immediate consequence of the following

\begin{theorem}
\label{sec:exist-nonex-neumann_withp}
Suppose that $\Omega \subset \R^{\sN}$ is a bounded domain, $2<p<2^*$
if $N\geq 3$ and $2<p<\infty$ if $N\leq 2$. Suppose moreover that
the matrix $B= (\beta_{ij})_{ij} \in S(n)$ is {\em not} strictly
copositive but satisfies $\beta_{ii} \ge 0$ for $i=1,\dots,n$.
Then the Neumann problem
\begin{equation}
  \label{eq:neumann-problem_withp}
\left \{
  \begin{aligned}
  -& \Delta u_i = \sum_{j=1}^n \beta_{ij}u_j^\frac{p}{2} u_i^{\frac{p}{2}-1},\quad\;u_1,\dots,u_n \ge 0&&\qquad
\text{in $\Omega$},\\
&\frac{\partial u_1}{\partial \nu}=\frac{\partial u_2}{\partial \nu}=\dots=\frac{\partial u_n}{\partial \nu}=0 &&\qquad \text{on $\partial \Omega$},
  \end{aligned}
\right.
\end{equation}
admits a nontrivial solution.
\end{theorem}

We briefly outline the proof of Theorem
\ref{sec:exist-nonex-neumann_withp} and point out the adjustments
which have to be made. Exactly as in the proof of Theorem \ref{sec:exist-nonex-neumann}, one can suppose without loss of generality that
\begin{equation}
  \label{eq:without-loss-1_withp}
B c \not= 0\qquad \text{for every $c \in
C^n_+ \setminus \{0\}$.}
\end{equation}
(otherwise we have a constant solution of the type \mbox{$u \equiv (\sqrt[p]{c_1^2},\dots,\sqrt[p]{c_n^2})$}). Moreover, we may also assume that
$\beta_{ii}>0$ for $i=1,\dots,n$ and
\begin{equation}
  \label{eq:boundary-positive_withp}
\sum_{i,j=1}^n \beta_{ij} c_i c_j >0 \qquad \text{for all $c \in \partial C^n_+  \setminus \{0\}$}.
\end{equation}
Now, we consider the functional
$$
E_p: \cH \to \R, \qquad E(u)= \frac{1}{2}\int_\Omega \bigl(|\nabla
u|^2+|u^-|^2
\bigr)\:dx -\phi_p(u),
$$
where
$$
\phi_p(u):=\frac{1}{p} \sum_{i,j=1}^n \int_\Omega \beta_{ij}
(u_i^+)^\frac{p}{2} (u_j^+)^\frac{p}{2}  \:dx \qquad \text{for
  $u \in \cH$}.
$$
Again we have that $E_p$ satisfies the Palais-Smale condition and that critical points of $E_p$
are nonnegative solutions of
\eqref{eq:neumann-problem_withp}. Choosing now $d \in \innt(C^n_+)$
such that
$b(d^\frac{p}{2})=b(d_1^\frac{p}{2},\ldots,d_n^\frac{p}{2})<0$ and defining
$$
M_\lambda := \{u \in \cH\::\: \|u\|=\lambda,\: \cL u \in \R d\},
$$
we can prove that
\begin{equation*}
\max\{-\phi(u),\Bigl(\int_\Omega |\nabla u|^2 +|u^-|^2\bigr) \:dx\Bigr)^\frac{p}{2}\} \ge \kappa_1
\lambda^p \qquad \text{ and } \qquad \phi(u) \le \kappa_2 \|u\|^p = \kappa_2\lambda^p
\end{equation*}
for all $u \in M_\lambda$ with constants $\kappa_1,\kappa_2>0$. From
this we then deduce that
$$
\sigma_\lambda:=\inf_{M_\lambda} E_p>0  \qquad \text{ for } \lambda=\frac{
\kappa_1^\frac{2}{p(p-2)}}{(4\kappa_2)^\frac{1}{p-2}}.
$$
The rest of the proof, namely the minimax principle relying on the
construction of the set $D \subset
C^n_+$ and the map $\Theta:D\to \cH$, can be carried out exactly as in
the special case $p=4$, see Section~\ref{sec:proof-exist-result-3}.

Next, we turn to the nonexistence results. Having once more the
results of Gidas \cite{gidas:80} in mind, we start by generalizing the
notion of strict cubic copositivity. Therefore in the following we will call a matrix $B \in S(n)$ {\em strictly ($p-1$)--copositive} if there exists
$\mu \in C^n_+$ such that
\begin{equation}
  \label{eq:2_withp}
\sum \limits_{i,j=1}^n \beta_{ij}c_j^\frac{p}{2}c_i^{\frac{p}{2}-1} \mu_i> 0 \qquad \text{for all
$c \in C^n_+ \setminus \{0\}$}.
\end{equation}

This notion gives rise to the following nonexistence result for (\ref{eq:neumann-problem_withp}).
\begin{proposition}
\label{(p-1)-copositive-nonex}
Suppose that $2<p\leq \frac{2N-2}{N-2}$ if $N\geq 3$ and $2<p<\infty$ if $N\leq2$. If $B \in S(n)$ is strictly $(p-1)$--copositive,  then
\eqref{eq:general-case} does not admit a nontrivial solution.
\end{proposition}

\begin{proof}
Let $\mu \in C^n_+$ be as in the definition above and take $\kappa=\kappa(B,\mu)>0$ such that
$$\sum_{i,j=1}^n \beta_{ij}c_j^\frac{p}{2} c_i^{\frac{p}{2}-1}\mu_i \geq \kappa \left(\sum_{i=1}^n \mu_i c_i \right)^{p-1}.$$
Suppose by contradiction that \eqref{eq:general-case} admits a nontrivial
solution $u=(u_1,\dots,u_n)$. Then the positive function $v:= \kappa^\frac{1}{p-2} \sum \limits_{i=1}^n \mu_i u_i$ satisfies
$$
-\Delta v= \kappa^\frac{1}{p-2} \sum_{i,j=1}^n \beta_{ij}u_j^\frac{p}{2} u_i^{\frac{p}{2}-1}\mu_i \ge
\kappa^{\frac{p-1}{p-2}}\Bigl(\sum \limits_{i=1}^n \mu_i u_i\Bigr)^{p-1} = v^{p-1} \qquad
\text{in $\R^N$.}
$$
By the result of Gidas \cite{gidas:80}, this is impossible since
$p-1\le \frac{N}{N-2}$ by assumption.
\end{proof}

Concerning the relationship between strict copositivity and strict
($p-1$)--copositivity, we have the following generalization of Proposition \ref{scc-implies-sc}

\begin{proposition}
\label{scc-implies-sc_withp}
Let $B \in S(n)$.
\begin{enumerate}
\item If $B$ is strictly $(p-1)$--copositive for some $p>2$, then it is also
strictly copositive.
\item If $n=2$ and $B$ is strictly copositive, then it is also
  strictly $(p-1)$--copositive for every $p>2$.
\end{enumerate}
\end{proposition}

\begin{proof}
i) Defining  $\tilde
b(c)=b(c^\frac{p}{2})=b(c_1^\frac{p}{2},\ldots, c_n^\frac{p}{2})$ for
$c\in C^n$, we can show similarly as in the proof of Proposition
\ref{scc-implies-sc} that $\tilde b$ is strictly positive on
$C^+_n\setminus\{0\}$, hence the same is true for $b$.\\
ii) Let $B=(\beta_{ij}) \in S(2)$ be strictly copositive, so that $\beta_{11},
\beta_{22}>0$ and $\beta_{12}> -\sqrt{\beta_{11} \beta_{22}}$.
 To show the
strict $(p-1)$--copositivity of $B$, we now consider $\mu:=(\frac{1}{\sqrt[p]{\beta_{11}}},\frac{1}{\sqrt[p]{\beta_{11}}})
        \in C^n_+$. Take an arbitrary $c=(c_1,c_2) \in
        C^2_+\setminus \{(0,0)\}$. If either $c_1=0$ or $c_2=0$ then $\sum_{i,j=1}^2 \beta_{ij} c_j^{\frac{p}{2}} \mu_i c_i^{\frac{p}{2}-1}$ is either equal to $\beta_{22}c_2^{p-1} \mu_2>0$ or $\beta_{11}c_1^{p-1} \mu_1>0$ respectively. Suppose now that $c_1,c_2\neq 0$ and put $\tilde c_i=c_i \sqrt[p]{\beta_{ii}}$ for $i=1,2$.
        Then
\begin{align*}
&\sum_{i,j=1}^2 \beta_{ij} c_j^{\frac{p}{2}} \mu_i c_i^{\frac{p}{2}-1}
=
\frac{1}{\sqrt[p]{\beta_{11}}}\ (\beta_{11} c_1^{p-1}+ \beta_{12}
    c_2^{\frac{p}{2}}c_1^{\frac{p}{2}-1}) + \frac{1}{\sqrt[p]{\beta_{22}}} (\beta_{22} c_2^{p-1}+ \beta_{12}
    c_1^{\frac{p}{2}} c_2^{\frac{p}{2}-1})
\\
&=\tilde c_1^{p-1} +\tilde c_2^{p-1} +
\frac{\beta_{12}}{\sqrt{\beta_{11}\beta_{22}}} (\tilde
c_2^{\frac{p}{2}} \tilde c_1^{\frac{p}{2}-1}+
\tilde c_1^{\frac{p}{2}} \tilde c_2^{\frac{p}{2}-1})>\tilde c_1^{p-1} +\tilde c_2^{p-1} -(\tilde
c_2^{\frac{p}{2}} \tilde c_1^{\frac{p}{2}-1}+
\tilde c_1^{\frac{p}{2}} \tilde c_2^{\frac{p}{2}-1})\geq 0,
\end{align*}
as required.
\end{proof}

By combining Theorem~\ref{sec:exist-nonex-rsn-1_withp} with Propositions
\ref{(p-1)-copositive-nonex} and \ref{scc-implies-sc_withp}, we
immediately get the following

\begin{corollary}
\label{coro:complete_result 1_withp}
Take $N\in \N$ and $2<p\leq \frac{2N-2}{N-2}$ if $N\geq 3$,  $2<p<\infty$ if $N\leq2$. Let $n=2$. If $\beta_{11},\beta_{22}$ are nonnegative, then the system
(\ref{eq:general-case}) admits a nontrivial solution if and only if
$B$ is not strictly copositive, i.e., if one of the strict
inequalities
$$
\beta_{11}>0,\quad \beta_{22}>0, \quad \beta_{12}>-\sqrt{\beta_{11}
  \beta_{22}}
$$
is not satisfied.
\end{corollary}

Generalizing Theorem \ref{planar-case}, we can also derive sharp nonexistence results for the
case of $n \ge 3$ components in dimensions $N=1,2$. However, we have to restrict our attention to the case
$2<p\leq 4$, and the proof is somewhat more complicated than in the
case $p=4$.

\begin{theorem}
\label{planar-case_withp}
If $N\leq 2$, $2<p\leq 4$ and $B\in S(n)$ is strictly copositive, then \eqref{eq:general-case} does not admit a nontrivial solution.
\end{theorem}
\begin{proof}
Suppose by contradiction that \eqref{eq:general-case} admits a nontrivial solution. Without loss of generality, we may assume that
$u_i>0$ in $\R^N$ for $i=1,\dots,n$. For $R>0$, consider the function $\phi_R$ defined in the proof of Theorem \ref{planar-case}. We recall that for $N=1,2$ we have
$$
\int_{\R^N}|\nabla \phi_R|^2\:dx\to 0 \qquad  \text{as $R \to \infty$.}
$$
Observe moreover that $|\nabla \phi_R|\leq C$ for some $C>0$ independent of $R>0$.

For $p=4$ the result of Theorem \eqref{planar-case_withp} is exactly the content of Theorem \ref{planar-case}.
Fix $2<p< 4$ and let $a\in(2,+\infty)$ be such that $p=\frac{4a}{2a-2}$. Multiplying (\ref{eq:special-case}) with
$\phi_R^a u_i^{1-\frac{p}{2}}$ and integrating by parts, we get
\begin{eqnarray*}
\sum_{j=1}^n\beta_{ij}\int_{\R^N} u_j^\frac{p}{2} \phi_R^a\:dx&=& \int_{\R^N} -\Delta u_i \phi_R^a u_i^{1-\frac{p}{2}} \:dx=
\int_{\R^N} \nabla u_i \cdot \nabla \Bigl(\phi_R^a u_i^{1-\frac{p}{2}} \Bigr)\:dx\\
&=& \int_{\R^N} \nabla u_i\cdot \Bigl( a\phi_R^{a-1}u_i^{1-\frac{p}{2}}\nabla \phi_R - \Bigl(\frac{p-2}{2}\Bigr)\phi_R^a u_i^{-\frac{p}{2}} \nabla u_i  \Bigr) \:dx\\
&=& -\int_{\R^N} \Bigl| \sqrt{\frac{p-2}{2}}\phi_R^\frac{a}{2}u_i^{-\frac{p}{4}}\nabla u_i - \frac{a}{2}\sqrt{\frac{2}{p-2}}\phi_R^{\frac{a}{2}-1}u_i^{1-\frac{p}{4}}\nabla \phi   \Bigr|^2\:dx +\\
&&+ \frac{a^2}{2(p-2)}\int_{\R^N}\phi_R^{a-2}u_i^\frac{4-p}{2}|\nabla \phi|^2\:dx\\
&\leq& \frac{a^2}{2(p-2)}\int_{\R^N}\phi_R^{a-2}u_i^\frac{4-p}{2}|\nabla \phi|^2\:dx.
\end{eqnarray*}
Next we let $c_j(R):= \int_{\R^N} u_j^\frac{p}{2} \phi_R^a\:dx$ for $j=1,\dots,n$. By multiplying the above inequality with $c_i(R)$ and
summing over $i$, we obtain from the strict copositivity of the matrix $B$
\begin{multline*}
0 \le \kappa_1 \sum_{i=1}^n c_i^2(R) \le \sum_{i,j=1}^n \beta_{ij}c_i(R) c_j(R) \le \frac{a^2}{2(p-2)} \sum_{i=1}^n c_i(R) \int_{\R^N}\phi_R^{a-2}u_i^\frac{4-p}{2}|\nabla \phi|^2\:dx \leq \\
\leq \frac{a^2}{2(p-2)} \Bigl(\max_{1\leq i\leq n}c_i(R) \Bigr) \sum_{i=1}^n  \int_{\R^N}\phi_R^{a-2}u_i^\frac{4-p}{2}|\nabla \phi|^2\:dx
\end{multline*}
and hence, by Young's inequality,
\begin{multline*}
\kappa_2 \sum_{i=1}^n \int_{\R^N}u_i^\frac{p}{2}\phi_R^a\:dx \leq \sum_{i=1}^n \int_{\R^N} \phi_R^{a-2} u_i^\frac{4-p}{2}|\nabla \phi_R|^2\: dx\leq\\
\leq \frac{\kappa_2}{2}\sum_{i=1}^n \int_{\R^N} u_i^\frac{p}{2}\phi_R^\frac{p(a-2)}{4-p}\:dx + \sum_{i=1}^n \kappa_3 \int_{\R^N} |\nabla \phi_R|^\frac{2p}{2p-4} \:dx.
\end{multline*}
Since $\frac{p(a-2)}{4-p}=a$ and $\frac{2p}{2p-4}\geq 2$, we finally get
$$\frac{\kappa_2}{2}\sum_{i=1}^n\int_{\R^N} u_i^\frac{p}{4}\phi_R^a\:dx \leq \kappa_3 \int_{\R^N}|\nabla \phi_R|^2\: dx\to 0  \qquad \text{ as }R\to +\infty.$$
Thus $\int_{B_R(0)} u_i^\frac{p}{2}\,dx \to 0$ as $R\to +\infty$ and hence $u_i\equiv 0$ for every $i$, contrary to what we have assumed.
\end{proof}

By combining Theorems \ref{sec:exist-nonex-rsn-1_withp} and \ref{planar-case_withp} we obtain the following result.

\begin{corollary}\label{coro:complete_result 2_withp}
Let $N\leq 2$, $2<p\leq 4$ and let $B\in S(n)$, $n\geq 2$, be such that $\beta_{ii}\geq 0$. Then the system \eqref{eq:special-case} admits a nontrivial solution if and only if $B$ is not strictly copositive.
\end{corollary}

Finally, for a general dimension $N$, we have an easy sufficient condition to check the strict ($p-1$)--copositivity of a matrix, and hence also a general sufficient condition for the nonexistence of solutions of \eqref{eq:general-case}.

\begin{proposition}
\label{sec:introduction--1_withp}
Suppose that
\begin{equation}\label{eq:cubic_copositivity_criterium}
\beta_{ii}> 0 \qquad \text{and}\qquad \sum_{{j=1}\atop{j
    \not=i}}^n\beta_{ij}^- >  -\beta_{ii} \qquad \text{for
  $i=1,\dots,n$,}
\end{equation}
where $\beta_{ij}^-=\min\{\beta_{ij},0\}$.
Then $B$ is strictly $(p-1)$--copositive for every $p>2$. In particular if \eqref{eq:cubic_copositivity_criterium} holds and either $2<p\leq \frac{2N-2}{N-2}$ with $N\geq 3$, or $2<p<\infty$ and $N\leq2$, then
\eqref{eq:special-case} does not admit a nontrivial
solution by Proposition~\ref{(p-1)-copositive-nonex}.
\end{proposition}

\begin{proof}
Using the simple inequality $s^\frac{p}{2} t^{\frac{p}{2}-1} +s^{\frac{p}{2}-1} t^\frac{p}{2} \le s^{p-1}+t^{p-1}$ for $s,t
\ge 0$ and the fact that $B=(\beta_{ij})_{ij}$ is symmetric, we obtain, by arguing exactly as in the proof of Proposition~\ref{sec:introduction--1}, that
$$\sum \limits_{i,j=1}^n \beta_{ij}c_j^\frac{p}{2} c_i^{\frac{p}{2}-1} \ge \kappa_0 \sum
\limits_{i=1}^n
c_i^{p-1} >0 \qquad $$
for every $c \in C^n_{+}\setminus\{0\}$ and $\kappa_0:= \min \limits_{i=1,\dots,n}\Bigl(\beta_{ii}+ \sum_{{j \not=i}}\beta_{ij}^-\Bigr) >0$.
\end{proof}

\section{Appendix}

Here we prove that the matrix
$$
B_\eps= \left (
\begin{array}{ccc}
1&-1+\eps&-1+\eps\\
-1+\eps &1&1\\
-1+\eps &1&1
\end{array}
\right), \qquad \eps>0.
$$
(see (\ref{eq:4})) is not strictly cubically copositive for $\eps>0$
small. This was claimed in Section~\ref{sec:introduction}. Suppose by contradiction that there
exists a sequence of positive numbers $\eps_k \to 0$ such that
$B_{\eps_k}$ is strictly cubically copositive for all $k$. Then there exists $(\mu_1^k,\mu_2^k,\mu_3^k)\in \innt(C^3_+)$ such that
\begin{multline}\label{eq:strict copitivity does not imply strict cubically positive}
\mu_1^k \left( c_1^3 + (\eps_k-1)c_1c_2^2 + (\eps_k-1)c_1 c_3^2\right) + \mu_2^k \left((\eps_k-1) c_1^2c_2 + c_2^3 + c_2 c_3^2\right)+ \\
+ \mu_3^k\left( (\eps_k-1)c_1^2 c_3 + c_2^2 c_3 + c_3^3  \right)>0
\end{multline}
for every $c\in C^3_+\setminus\{0\}$. By dividing the previous
inequality by
$\mu_1^k>0$, we can suppose without loss of generality that $\mu_1^k=1$. We show that
\begin{equation}\label{eq: mu_2^eps , mu_3^eps to 1}
\mu_2^k\to 1,\quad \mu_3^k \to 1 \qquad \text{ as }k  \to \infty.
\end{equation}
Taking $c=(\lambda,1,0)\in C^3_+$ in \eqref{eq:strict copitivity does
  not imply strict cubically positive} with $\lambda>0$, we obtain
$$
(\lambda^3+(\eps_k-1)\lambda)+\mu_2^k ((\eps_k-1)\lambda^2+1)>0 \quad
\text{ for every } k \in \N
$$
and hence
$$
(\lambda^3-\lambda)+\liminf_{k \to \infty} \Bigl[\mu_2^k (1-\lambda^2)\Bigr] \ge 0.
$$
As a consequence, we get
$$
\liminf_{k \to \infty} \mu_2^k \ge \lambda \quad \text{for all
  $\lambda \in (0,1)$}\qquad \text{and} \qquad
\limsup \limits_{k \to \infty} \mu_2^k \le \lambda \quad \text{for all
$\lambda>1$,}
$$
which eventually yields $\lim \limits_{k \to \infty}\mu_2^k= 1.$
Considering $c=(\lambda,0,1)$ in (\ref{eq:strict copitivity does not
  imply strict cubically positive}), a similar argument shows that $\lim
\limits_{k \to \infty}\mu_3^k= 1.$ By combining \eqref{eq:strict copitivity does not imply strict cubically positive} with \eqref{eq: mu_2^eps , mu_3^eps to 1}, we conclude that
$$
(c_1^3-c_1 c_2^2-c_1 c_3^2) + (-c_1^2 c_2 + c_2^3 +c_2 c_3^2)+(-c_1^2
c_3 + c_2^2 c_3 + c_3^3)\geq 0 \quad \text{for all $c\in C^3_+$,}
$$
which is false when evaluated at $c=(3,2,2)$.
\\[15pt]


\noindent {\bf Acknowledgments.}
H. Tavares was supported by FCT, grant SFRH/BD/28964/2006 and Financiamento Base 2008 - ISFL/1/209.\\
A large part of the paper was written while H. Tavares was visiting the University of Frankfurt. His stay was also partially supported by a grant of the Justus-Liebig-University, Giessen.


\newcommand{\noopsort}[1]{} \newcommand{\printfirst}[2]{#1}
  \newcommand{\singleletter}[1]{#1} \newcommand{\switchargs}[2]{#2#1}

\noindent \verb"htavares@ptmat.fc.ul.pt"\\
University of Lisbon, CMAF, Faculty of Science, Av. Prof. Gama Pinto
2, 1649-003 Lisboa, Portugal

\noindent \verb"susanna.terracini@unimib.it"\\
Dipartimento di Matematica e Applicazioni, Universit\`a degli Studi
di Milano-Bicocca, via Bicocca degli Arcimboldi 8, 20126 Milano,
Italy

\noindent \verb"gianmaria.verzini@polimi.it"\\
Dipartimento di Matematica, Politecnico di Milano, p.za Leonardo da
Vinci 32,  20133 Milano, Italy

\noindent \verb"weth@math.uni-frankfurt.de"\\
Institut f\"ur Mathematik, Goethe-Universit\"at Frankfurt, Robert-Mayer-Str. 10,
D-60054 Frankfurt a.M., Germany

\end{document}